\documentclass[12pt,a4paper]{article}

\usepackage[utf8]{inputenc}
\usepackage[english]{babel}
\usepackage{multirow}
\usepackage[toc,page]{appendix}
\usepackage{colortbl}
\usepackage{natbib,color}
\usepackage{array, blkarray, makecell}
\usepackage[left=2.54cm,right=2.54cm,
    top=2.54cm,bottom=2.54cm,bindingoffset=0cm]{geometry}
    \usepackage{graphicx}
      \usepackage{amsmath,amssymb,euscript,amsthm,amsfonts,mathrsfs,amscd}
\usepackage{color}
\usepackage{amsthm}
\usepackage{rotating}
\newtheorem{theorem}{Theorem}

\numberwithin{equation}{section}
\usepackage[affil-it]{authblk}
\def \bv {\textbf{v}}
\def \bZ {\textbf{Z}}
\def \bz {\textbf{z}}
\def \bp {\textbf{p}}
\def \bx {\textbf{x}}
\def \bJ {\textbf{J}}
\usepackage{array, blkarray, makecell}
\newcolumntype{C}{>{\centering\arraybackslash} m{1.25cm} }  %# New column type
\newcolumntype{Q}{>{\centering\arraybackslash} m{0.6cm} }  %# New column type
\newcolumntype{W}{>{\centering\arraybackslash} m{1cm} }  %# New column type
 \title{An information-theoretic approach for selecting arms in clinical trials}
  \author{Pavel Mozgunov and Thomas Jaki \hspace{.2cm}\\
     Department of Mathematics and Statistics, \\ Lancaster University, Lancaster, UK}
\date{}

\usepackage{algpseudocode}
\usepackage{algorithm}
\theoremstyle{plain}
\newtheorem*{theorem*}{Probability integral transform}
\newtheorem*{col*}{Corollary of the probability integral transform}
\usepackage{mathtools}

\begin{document}
\maketitle

\begin{abstract}
{The question of selecting the ``best'' amongst different choices is a common problem in statistics. In drug development, our motivating setting, the question becomes, for example: which treatment gives the best response rate. Motivated by a recent development in the theory of context-dependent information measures, we propose an experimental design based on a simple and intuitive criterion to govern arm selection in  an experiment with multinomial outcomes. The criterion leads to reliable selection of the correct arm without any parametric or monotonicity assumptions. The asymptotic properties of the design are studied for different allocation rules and the small sample size behaviour is evaluated in simulations in the context of Phase I and Phase II clinical trials.  We compare the proposed design to currently used alternatives and discuss its practical implementation. \\}

\end{abstract}

\section{Introduction}
\label{sec1}

Over the past decades a variety of different methods for clinical trials aiming to select the ``optimal'' arm (e.g. dose, treatment, combination of treatments,...) have been proposed in the literature \citep{QPF90,wages2011,mams,villar2015,lee2016decision}. Given $m$ arms, the aims of Phase I and Phase II clinical trials are often to identify the target arm (TA) corresponding to the toxicity probability closest to the target $0<\gamma_t<1$ and/or the efficacy probability closest to the target $0<\gamma_e<1$. Despite the similar problem formulation for Phase I (evaluating toxicity) and Phase II (evaluating efficacy) trials, quite different approaches are generally utilized. 

In Phase I dose-escalation trials, model-based designs assuming a monotonic dose-toxicity relationship are shown to have good operating characteristics in the setting of a single cytotoxic drug \citep{iasonos2016}.
%Flexible curve-free alternatives that are based on a known ordering of doses with respect to the toxicity risk have been also proposed \citep{whitehead2010,mtpi}.
 The ability to find the TA using these methods is, however, rather limited if the assumption of monotonicity is not satisfied \citep{shen1996}. While this is not common for cytotoxic drugs, the uncertainty about toxicity and efficacy orderings holds for drugs combinations, dose-schedules and molecularly targeted drugs \citep{wages2011,lagarde2015}. 

To overcome the issue of an unknown ordering in the context of Phase I trials, some specialised approaches for combination and dose-schedule trials have been proposed \cite[e.g.][]{Thall2003,wages2011,guo2016}.
The common features of the majority of novel Phase I methods relaxing the monotonicity assumption is either relying on a parametric model or on explicit order of toxicity/efficacy. While such methods allow to borrow information across treatment arms, they might fail to find the TA if the model or ordering is misspecified. There is also growing interest in advanced trials with a large number of potential orderings where specifying all of them (or a corresponding parametric model) is not feasible \citep[see e.g.][for an example]{mozgunov2018}. In addition, more complex outcomes than a simple binary endpoint are becoming more frequent in dose-finding trials \citep[see e.g.][]{yuan2007,lee2010continual,lee2017dose} as they can carry more information about the drug's mechanism of action. Despite this, methods for studies with complex outcomes that do not require monotonicity or a complex model are sparse to date.

Thinking more broadly about selecting one or more arms during a trial (the main objective of many Phase II studies), different methods that consider arms being independent have been proposed \citep[see e.g.][]{stallard2003,koenig2008,whitehead2009,mams}.  \cite{williamson2016} have recently advocated designs maximising the expected number of responses in small populations trials. As a result, adaptive randomisation methods and optimal multi-arm Bandit (MAB) approaches are starting to be considered  more commonly in Phase II clinical trials. Although, MAB designs outperform other well-established methods of randomization (e.g. fixed randomization) in terms of expected number of successes, they can suffer low statistical power for testing comparative hypotheses. This problem corresponds to the `exploration vs exploitation` trade-off \citep{azriel2011} and some rule-based modifications have been proposed archive a better balance of the two objectives \citep[see e.g.][]{villar2015, williamson2016}. However, the majority of MAB approaches consider univariate binary response only and hone in on the arm with the largest effect by default and cannot be applied to clinical trials which aim to select the TA corresponding to the target probability $\gamma_e$ where $\gamma$ is often between $(0.7,1)$.  On the other hand, model-based alternatives suffer from the problem of model misspecification. Although, some of these challenges can be overcome by methods such as MCPmod \citep{mcpmod}, we believe that flexible alternatives that can be potentially used in the described settings are required.

Overall, the research problems described above can be considered as the general issue of correct identification of the TA whose response probability is closest to the percentile $0<\gamma<1$ or equivalently in the multidimensional case, whose characteristics are closest to the vector $\boldsymbol{\gamma} \in \mathbb{S}^d$ where $\mathbb{S}^d$ is a $d$-dimensional unit simplex defined in (\ref{simplex}). We propose a general experimental design for studies with multinomial outcomes to solve this generic problem. Based on developments in the information theory of context-dependent measures of information \citep{belis1968,kelbert2015,suhov2016}, we derive a criterion which governs arm selection in the experiment. The criterion is based on the maximisation of the information gain when considering an experiment with a particular interest in arms whose response probabilities are in the neighbourhood of $\boldsymbol{\gamma}$. Recently other designs using the information gain principle have been proposed \citep[see e.g.][]{barrett2016,kim2016}. In contrast to these methods, the proposed approach allows incorporation of the context of the outcomes (e.g. avoid high toxicity or low efficacy) in the information measures themselves. This is achieved by assigning a greater "weight" to the information obtained about arms with characteristics close to the desired level. Another difference to the majority of information-theoretic approaches is that the design is based on the so-called "patients' gain" which allocates each new patient to the treatment that is considered best while taking into account the uncertainty about the estimates for each arm. This leads to fulfilling of statistical goals of the experiment under the ethical constraints.

The proposed criterion is not restricted to a particular model and can be used, for examples, to govern selection within traditional parametric designs. However, motivated by relaxing parametric and monotonicity assumptions, we demonstrate that good operating characteristics of the design can be achieved without employing these assumptions. For the special case of a complex combination-schedule clinical trial the proposed design has already been shown to be superior compared to currently employed methods and to be practically applicable to an ongoing clinical trial \citep{mozgunov2018}. In this work, we generalise the approach for an arbitrary number of outcomes and general class of "weight" functions, study the asymptotic behaviour of the design and compare the performance to the currently used method in contexts of Phase I and Phase II trials.

%propose an approximation method for the original design to facilitate otherwise computationally extremely demanding evaluations of operating characteristics via simulations. 

%{Given that space is an issue I would cut the rest of this paragraph} We study the small-sample behaviour of the proposed design in the context of Phase I clinical trial with unknown non-monotonic arm-toxicity relation, discuss the practical implementation such as a safety constraint and compare the performance to the currently used method. In the comparison we include methods relaxing monotonicity assumption, but also consider the design based on the monotonicity assumption to illustrate the consequence of the misspecified model. We also apply the novel design to the context of a Phase II clinical trial in small populations with two-fold goal to achieve high statistical power and to assigned as much as possible patients to superior arm. It is shown that the design is able to find the superior arm even without specifying an accurate target value $\gamma$. We demonstrate how the tuning parameter of the trade-off $\kappa$ allows an investigator to benefit in terms of statistical power compared to MAB approaches and in terms of the number of treated patients compared to equal randomisation. We also discuss the practical implementation of the novel design in contexts of actual clinical trials. 

The remainder of the paper is organized as follows: derivations of the criterion and assignment rules are given in Section \ref{sec:methods}. The asymptotic behaviour is studied in Section \ref{sec:asymptotics}. Section \ref{sec:examples} presents illustrative examples of the design together with a comparisons to alternative methods. We conclude with a discussion in Section \ref{sec:discussion}.

\section{Methods\label{sec:methods}}
The novel design is based on the maximisation of the information gain in the experiment with an interval of the specific interest (neighbourhood of $\boldsymbol{\gamma}$). Below, we derive an explicit formula for the information gain in the context of a trial with multinomial outcomes.

\subsection{Information-theoretic concepts}
Consider a discrete random variable taking one of $d$ values and a corresponding random probability vector $\bZ = \left[Z^{(1)},Z^{(2)},\ldots,Z^{(d)}\right] \in \mathbb{S}^d$ defined on a unit simplex 
\begin{equation}
 \mathbb{S}^d~=~\{\bZ:~Z^{(1)}~>~0,Z^{(2)}~>~0,~\ldots,Z^{(d)}~>~0;~\sum_{i=1}^d~Z^{(i)}=1\}.
\label{simplex}
\end{equation}
Assume that $\bZ$ has a prior Dirichlet distribution ${\rm Dir} (\bv + \bJ)$ where $\bv= \left[v^{(1)},\ldots,v^{(d)}\right]^{\rm T} \in \mathbb{R}_{+}^{d}$, $\sum_{i=1}^{d}v^{(i)}=\beta$ and $\bJ$ is a $d$-dimensional unit vector. After $n$ realizations of a discrete random variable in which $x^{(i)}$ outcomes of $i$ are observed, $i=1,\ldots,d$, the random vector $\bZ_n$ has a Dirichlet posterior distribution with density function
\begin{equation}
f_n(\bp|\bx) = \frac{1}{B(\bx+\bv + \bJ)}\prod_{i=1}^{d}\left(p^{(i)}\right)^{x^{(i)}+v^{(i)}}, \ B(\bx+\bv + \bJ) = \frac{\prod_{i=1}^d \Gamma (x^{(i)}+v^{(i)}+1)}{\Gamma \left(\sum_{i=1}^{d}(x^{(i)}+v^{(i)}+1) \right) }
\label{pdf}
\end{equation}
where $\bp = \left[p^{(1)},\ldots,p^{(d)} \right]^{\rm T}$, $\bx = \left[x^{(1)},\ldots,x^{(d)} \right]$, $\sum_{i=1}^{d}x^{(i)}=n$, $0<p^{(i)}<1$, $\sum_{i=1}^{d}p^{(i)}=1$ and
$B(\bx+\bv + \bJ)$
is the Beta-function and $\Gamma(x)$ is the Gamma-function.

Let $\boldsymbol{\alpha}= \left[\alpha^{(1)},\ldots,\alpha^{(d)} \right]^{\rm T} \in \mathbb{S}^d$  be the vector in the neighbourhood of which $f_n$ concentrates as $n \to \infty$. A classic question of interest in this setting is to estimate the probability vector, $\boldsymbol{\alpha}$. The information required to answer the estimation question can be measured by the Shannon differential entropy of $f_n$ \citep{cover2006} 
\begin{equation}
h(f_n) = - \int_{\mathbb{S}^d}  f_n(\bp|\bx) {\rm log} f_n(\bp|\bx) {\rm d}\bp
\label{entropy}
\end{equation}
with convention $0 \log 0 =0$. The classic formulation of the estimation question, however, does not take into account the fact that an investigator would like to find the target arm (TA) having pre-specified characteristics $\boldsymbol{\gamma}=\left[ \gamma^{(1)},\ldots,\gamma^{(d)} \right] \in \mathbb{S}^d$. It does not also reflect that one would like to have more precise estimation about the vector $\boldsymbol{\alpha}$ for those arms only which have characteristics close to $\boldsymbol{\gamma}$. This is a consequence of the fact that the classic information measures do not depend on the nature of the outcomes $\bp$, but on the probability of the corresponding event $f(\bp)$ and therefore are called context-free. While it gives the notion of information great flexibility which explains its successful application in various fields, the context-free nature might be considered as a drawback in many application areas as it would be demonstrated below.

To take into account the context of the experiment and the nature of the outcomes $\bp$, one can consider an estimation experiment with "sensitive" area (i.e. the neighbourhood of $\boldsymbol{\gamma}$). The information required in such an experiment can be measured by the \textit{weighted} Shannon differential entropy \citep{belis1968,clim2008,kelbert2015,suhov2016} of $f_n$ with a positive weight function $\phi_n(\bp)$
\begin{equation}
h^{\phi_n}(f_n) = - \int_{\mathbb{S}^d}  \phi_n (\bp) f_n(\bp|\bx) {\rm log} f_n(\bp|\bx) {\rm d}\bp.
\label{weighted_entropy}
\end{equation}
The crucial difference between the information measures given in Equation (\ref{entropy}) and Equation (\ref{weighted_entropy}) is the weight function, $\phi_n(\bp)$, which emphasizes the interest in the neighbourhood of $\boldsymbol{\gamma}$ rather than on the whole $\mathbb{S}^d$. It reflects that the information about the probability vector which lies in the neighbourhood of $\boldsymbol{\gamma}$ is more valuable in the experiment.

Due to the limited sample size in an actual studies, an investigator is typically interested in answering the question: Which arm has an associated probability vector closest to $\boldsymbol{\gamma}$ while ensuring accurate estimation of the probability vector for the TA only. For this question, the information gain from considering the experiment with sensitive area equals to
\begin{equation}
\Delta_n=h(f_n)-h^{\phi_n}(f_n).
\label{gain}
\end{equation}
Following the information gain approach, the first term in the equation above is the information in a classic experiment using context-free measure, while the second (novel) term is the information when the context of events is taken into account. Alternatively, $\Delta_n$ can be considered as an average amount of the additional statistical information required when considering the context-dependent estimation problem instead of the traditional one. 

%A quantitative measure of the information required for question (ii) alone is the difference of the weighted differential entropy $h^{\phi_n}(f_n)$ and the differential entropy $h(f_n)$.

The information gain in Equation (\ref{gain}) requires specification of a weight function which defines the "value" of the information in different areas of the simplex $\mathbb{S}^d$. To track the influence of the weight function explicitly we consider a weight function in the Dirichlet form:
\begin{equation}
\phi_n(\bp) = C(\bx,\boldsymbol{\gamma},n) \prod_{i=1}^{d}\left(p^{(i)}\right)^{\gamma^{(i)} n^{\kappa}}
\label{weight}
\end{equation}
where $\kappa \in (0,1)$ is a parameter and $C(\bx,\boldsymbol{\gamma},n)$ is a constant which is chosen to satisfy the normalization condition $\int_{\mathbb{S}^d} \phi_n (\bp) f_n(\bp|\bx){\rm d} \bp =1.$ The parameter $\kappa$ is restricted to the unit interval to ensure asymptotically unbiased estimates of the vector $\boldsymbol{\alpha}$: $ \lim_{n \to \infty} \int_{\mathbb{S}^d} \bp \phi_n(\bp)f_n(\bp) {\rm d}\bp =\boldsymbol{\alpha}$. This emphasises the interest in the identification of the TA for the small and moderate sample sizes typical for many applications. The asymptotic behaviour of the information gain, $\Delta_n$, for the family of weight functions (\ref{weight}) is studied in Theorem \ref{main_theorem}.

\begin{theorem}
Let $h(f_n)$ and $h^{\phi_n}(f_n)$ be the standard and weighted differential entropies of (\ref{pdf}) with weight function (\ref{weight}). Let $\lim_{n \to \infty}  \frac{x^{(i)}(n)}{n}=\alpha^{(i)}$ for $i=1,2,\ldots,d$ and $\sum_{i=1}^{d}x^{(i)}=n$, then
$$
\Delta_n = O \left(\frac{1}{n^{1-2\kappa}} \right)  \ {\rm as} \ n \to \infty \ {\rm if} \ \kappa<\frac{1}{2};
$$
$$
\Delta_n  = - \frac{1}{2} \left( \sum_{i=1}^{d} \frac{\left(\gamma^{(i)} \right)^2}{\alpha^{(i)}} -1 \right) n^{2 \kappa -1} + \omega(\boldsymbol{\alpha},\boldsymbol{\gamma},\kappa,n)  + O \left(\frac{1}{n^{\eta (1-\kappa)-\kappa}} \right) \ {\rm as} \ n \to \infty  \ {\rm if} \ \kappa\geq\frac{1}{2}
$$
%\begin{eqnarray*}
%\lim_{n \to \infty}\Delta_n &=& 0 \ {\rm if} \ \kappa<\frac{1}{2},\\
%\lim_{n \to \infty} \Delta_n &=& - \frac{1}{2} \left( \sum_{i=1}^{d} \frac{\left(\gamma^{(i)} %\right)^2}{\alpha^{(i)}} -1 \right)  \ {\rm if} \ \kappa = \frac{1}{2}, \\
%\end{eqnarray*}
%$$
%\lim_{n \to \infty} \left( \Delta_n  + \frac{1}{2} \left( \sum_{i=1}^{d} \frac{\left(\gamma^{(i)} \right)^2}{\alpha^{(i)}} -1 \right) n^{2 \kappa -1} - \omega(\alpha,\gamma,\kappa,n)  \right) = 0 \ {\rm if} \ \kappa\geq\frac{1}{2}
%$$
where
$$\omega(\boldsymbol{\alpha},\boldsymbol{\gamma},\kappa,n) = \sum_{j=3}^{\eta} \frac{(-1)^{j-1}}{j}n^{j \kappa - j +1} \left( \sum_{i=1}^{d} \frac{\left(\gamma^{(i)}\right)^j}{\left(\alpha^{(i)}\right)^{j-1}} -1 \right) \ {\rm and} \ \eta=\lfloor \left({1-\kappa}\right)^{-1} \rfloor$$
%with convention $\sum_{j=3}^0 = \sum_{j=3}^1=\sum_{j=3}^2=0$.
\label{main_theorem}
\end{theorem}
\begin{proof}
The proof is given in the Appendix.
\end{proof}

The information gain, $\Delta_n$, tends to $0$ for $\kappa<1/2$ which implies that assigning a value of information with rate less than $1/2$ is insufficient to emphasize the importance of the context of the study. However, the limit is non-zero for $\kappa \geq 1/2$. Following the conventional information gain approach, one would like to make a decision which maximises the statistical information in the experiment. The information gain $\Delta_n$ is always non-positive and for any fixed $n$ its asymptotics achieves the maximum value~$0$ at the point $\alpha^{(i)}=\gamma^{(i)}$, $i=1,\ldots,d$  (all constants are cancelled out). Indeed, Theorem \ref{main_theorem} implies that when maximising the information gain $\Delta_n$, one tends to collect more information about the arm which has characteristics $\boldsymbol{\alpha}$ close to the target  $\boldsymbol{\gamma}$. To keep the trackable solution which can be easily interpreted in applications, we construct the arm selection criterion using the leading term of the asymptotic expression for $\Delta_n$ in Theorem \ref{main_theorem}:
% {this next sentence is a bit clumsy. I wonder if you can avoid talking about absolute values and the special case of $\kappa=1/2$ here by simply writing the $\kappa=1/2$ as part of the expression for larger $\kappa$. I.e. you would only have two cases in the theorem $\kappa<0.5$ and $\kappa\geq 0.5$ (I think it is true that the expression for >0.5 simpifies to the one for =0.5 if kappa=0.5 is inserted). Then you can simply define the criterion as the leading term of the asymptotic expression here. You can spell out the explicit case of kappa=0.5 underneath the theorem. } As the proposed limit is always non-negative, we would consider the absolute value of the limit of $\Delta_n$ for the convenience of the reader
\begin{equation}
\delta^{(\kappa)}(\boldsymbol{\alpha},\boldsymbol{\gamma}):=\frac{1}{2} \left( \sum_{i=1}^{d} \frac{\left(\gamma^{(i)}\right)^2}{\alpha^{(i)}} -1 \right)n^{2\kappa-1}.
\label{general_measure}
\end{equation}
Note that maximising the leading term of the information gain asymptotics is equivalent to minimising $\delta^{(\kappa)}(\boldsymbol{\alpha},\boldsymbol{\gamma})$. Equation (\ref{general_measure}) can be considered as the measure of the divergence between $\boldsymbol{\alpha}$ and $\boldsymbol{\gamma}$ and the criterion which governs the selection such that the information gain is maximised. The criterion (\ref{general_measure}) is intuitive as it reflects explicitly the fact that an investigator tends to collect more information about the arm with probability vector close to $\boldsymbol{\gamma}$, and also shares some desirable properties. Clearly, $\delta^{(\kappa)}(\cdot) \geq 0$ and $\delta^{(\kappa)}(\cdot)=0$ iff $\boldsymbol{\alpha}=\boldsymbol{\gamma}$ for all $\kappa$ and $n$. The boundary values $\alpha^{(i)}=0$, $i=1,\ldots,d$ correspond to infinite values of $\delta^{(\kappa)}(\boldsymbol{\alpha},\boldsymbol{\gamma})$ which is advocated by \cite{ait} as one of the important properties for functions defined on simplex~$\mathbb{S}^d$. We construct the design based on the selection criterion (\ref{general_measure}) below.

\subsection{Selection criterion}
Consider a discrete set of $m$ arms, $A_1,\ldots,A_m$, associated with probability vectors $\boldsymbol{\alpha}_1,\ldots,\boldsymbol{\alpha}_m$ and $n_1,\ldots,n_m$ observations.
Arm $A_j$ is optimal if it satisfies
$
\delta^{(\kappa)}(\boldsymbol{\alpha}_j,\boldsymbol{\gamma}) =\inf_{i=1,\ldots,m}\delta^{(\kappa)}(\boldsymbol{\alpha}_i,\boldsymbol{\gamma}).
$
To estimate $\delta^{(\kappa)}(\boldsymbol{\alpha}_i,\boldsymbol{\gamma})$  a random variable $\tilde{\delta}_{n_i}^{(\kappa)}  \equiv \delta^{(\kappa)}(\bZ_{n_i},\gamma) $ is introduced where $\bZ_{n_i}$ has the Dirichlet distribution given in (\ref{pdf}). Let us fix an arm $\bZ_{n_i} \equiv \bZ_{n}$ and denote $\tilde{\delta}_{n_i}^{(\kappa)} \equiv \tilde{\delta}_{n}^{(\kappa)}$. It is a known that a Dirichlet random variable (after appropriate normalization) weakly converges to a multivariate normal distribution. In fact, a stronger result can be shown using the Kullback-Leibler distance
$
\mathbb{D}(f\ ||\ g)= \int_{\mathbb{R}} f(x) {\rm log} \frac{f(x)}{g(x)} {\rm d}x
$
where $g$ and $f$ are probability density functions.
\begin{theorem}
\label{th_dirichlet}
Let $\widetilde{\bZ}_n = \Sigma^{-1/2}\left(\bZ_n-\boldsymbol{\alpha} \right)$ be a random variable with pdf $\widetilde{f}_n$ where pdf of $\bZ_n$ is given in (\ref{pdf}) with $\lim_{n \to \infty}  \frac{x^{(i)}(n)}{n}=\alpha^{(i)}$ for $i=1,2,\ldots,d$, $\sum_{i=1}^{d}x^{(i)}=n$ and where $\Sigma$ is a $d$-dimensional square matrix with elements $\Sigma_{[ij]} =  \frac{\alpha^{(i)} (1-\alpha^{(i)})}{n}$ if $i=j$ and $\Sigma_{[ij]} = -\frac{\alpha^{(i)} \alpha^{(j)}}{n}$ if $i \neq j$. Let $\overline{\bZ}$ be the multivariate Gaussian random variable $\mathcal{MN}\left(0,I_{d-1} \right)$ ($I_{d-1}$ is the $(d-1)$-dimensional unit matrix) with pdf $\varphi$  and the differential entropy $h(\varphi) = \frac{1}{2}{\rm log} \left(\left(2 \pi e \right)^{d-1} \right)$. Then the Kullback-Leibler divergence of $\varphi$ from $\tilde{f}_n$ tends to $0$ as $n \to \infty$
which implies that $\widetilde{\bZ}_n$ weakly converges to $\overline{\bZ}$.
\end{theorem}
\begin{proof}
The proof in given in the Appendix.
\end{proof}

Using Theorem \ref{th_dirichlet} the following result can be obtained for the proposed criterion.
\begin{theorem}
Under the assumptions of Theorem \ref{th_dirichlet}, let  $\tilde{\delta}_{n}^{(\kappa)}  = \delta^{(\kappa)}(\bZ_{n},\boldsymbol{\gamma})$,
$\nabla  \delta^{(\kappa)}(\bz,\boldsymbol{\gamma}) = \left[\frac{\partial \delta^{(\kappa)}(\bz,\boldsymbol{\gamma})}{\partial z^{(1)}}, \ldots, \frac{\partial \delta^{(\kappa)}(\bz,\boldsymbol{\gamma})}{\partial z^{(d)}} \right]^{\rm T}$,
$\bar{\delta}_{n}^{(\kappa)} = \bar{\Sigma}^{-1/2} \left(\delta^{(\kappa)}(\bZ_{n},\boldsymbol{\gamma}) - \delta^{(\kappa)}(\boldsymbol{\alpha},\boldsymbol{\gamma}) \right)$
where $\bar{\Sigma}= \nabla_{\boldsymbol{\alpha}}^{\rm T} \Sigma \nabla_{\boldsymbol{\alpha}}$ and $\nabla_{\boldsymbol{\alpha}} \equiv \nabla  \delta^{(\kappa)}(\bz,\boldsymbol{\gamma})$ evaluated at $\bz=\boldsymbol{\alpha}$. Let $\bar{Z}$ be a standard Gaussian RV. Then,
$\lim_{n \to \infty} \mathbb{E}\tilde{\delta}_{n}^{(\kappa)}=\delta^{(\kappa)}(\boldsymbol{\alpha},\boldsymbol{\gamma}), \ \ 
\lim_{n \to \infty} \mathbb{V}\tilde{\delta}_{n}^{(\kappa)}=0,$
and $\bar{\delta}_{n}^{(\kappa)}$ weakly convergences to $\bar{Z}$.
\label{prop}
\end{theorem}
%\begin{proof}
%Statements about expectation and variance are straightforward to prove by computing the first and the second moments of $\delta^{(\kappa)}(\bZ_{n},\boldsymbol{\gamma})$. The weak convergence result follows from Theorem \ref{th_dirichlet} and by applying the the Delta-method \citep{resnick2013}.
%\end{proof}

A single summary statistics for $\delta^{(\kappa)}(\bZ_{n},\boldsymbol{\gamma})$ is needed to select the most promising arm in the sequential experiment. While a Bayesian estimator can be used, we will focus on the intuitively clear and simple `plug-in` estimator, $\hat{\delta}^{(\kappa)}(\hat{\bp}_{n},\gamma) \equiv \hat{\delta}^{(\kappa)}_n$  with $\hat{\bp}_{n}= [\hat{p}_n^{(1)},\ldots, \hat{p}_n^{(i)},\ldots,\hat{p}_n^{(d)}]^{\rm}$ and $\hat{p}_{n}^{(i)} = \frac{x^{(i)}+v^{(i)}}{n+\beta^{(i)}}, \ i=1,\ldots,d$, the mode of the posterior Dirichlet distribution. The estimator for the arm $A_j$ takes the form
\begin{equation}
\hat{\delta}^{(\kappa)}_{n_j}=\delta^{(\kappa)}(\hat{\bp}_{n_j},\boldsymbol{\gamma}) = \frac{1}{2} \left( \sum_{i=1}^{d} \frac{\left(\gamma^{(i)}\right)^2}{\hat{p}_{n_j}^{(i)}} -1 \right)n_j^{2\kappa-1}, \ j=1,2,\ldots,m.
\label{plugin}
\end{equation}
Note that by Theorem \ref{prop} for any $\varepsilon >0$
$
\lim_{n_j\to \infty} \mathbb{P} \left(\tilde{\delta}_{n_j}^{(\kappa)} \in [\hat{\delta}_{n_j}^{(\kappa)}-\varepsilon,\hat{\delta}_{n_j}^{(\kappa)}+\varepsilon] \right) =1.
$ The statistics (\ref{plugin}) is used to govern the selection among arms during the experiment. Note that the estimator above requires a vector of prior parameters $\bv_j$, $j=1,\ldots,m$ to start the experiment. This choice implies an initial ordering in which an investigator would like to test the arms before the data is available.

%{Since you have not yet actually spelled out how you are going to use the criterion for allocating patients, the following sentences are in the wrong place. They would go after defining the actual procedure.}
% In the context of the experimental design, it follows that the criterion favours decisions further from the bounds and closer to the target. The second term, $n^{2 \kappa -1}$, reflects a penalty on the number of observations on the same arm (i.e. many observations on one arm mean that other arms are favoured). Thus, $\kappa=1/2$ corresponds to no penalty and increasing values of $\kappa$ correspond to greater interest in the statistical power of the experiment rather than treating maximum number of patients. Importantly, even for $\kappa<1/2$ the criterion dictates to allocate a greater proportion of patient to the optimal arm due to the first term in Equation (\ref{general_measure})

%Although, the above estimator does not use any parametric assumptions, it requires a vector of prior parameters $\bv_j$, $j=1,\ldots,m$. This choice implies a specific ordering to be followed before the experiment, but this ordering will change as a trial progresses.

\subsection{Specific Assignment rules}

The criterion (\ref{general_measure}) summarizes the arm's characteristics  and can be applied to different types of sequential experiments. We consider two assignment rules: Rule I which randomizes between arms and Rule II which selects the ``best'' arm. These rules follow the setting of the motivating clinical trials: Rule II is widely used in Phase I trials where the randomization to all doses is not ethical or in the typical MAB setting where the primary goal is to maximize the number of successes. Note that randomization (when is ethical) allows to decrease the probability of identifying a suboptimal arm \citep{thall2007}.

\subsubsection{Rule I: Randomization\label{sec:rule1}} 
Under Rule I, the arm used next in the experiment is randomised with probabilities
$
 \tilde{w}_j\equiv \frac{1/{\tilde{\delta}_{n_j}^{(\kappa)}}}{\sum_{i=1}^m 1/{\tilde{\delta}_{n_i}^{(\kappa)}}}, \ \ j=1, \ldots, m
$
and from Theorem \ref{prop} 
$
{w}_j = \lim_{n_1,n_2,\ldots,n_m \to \infty} \mathbb{E}\left( \tilde{w}_j \right) =  \frac{1/{{\delta}^{(\kappa)}_{n_j}}}{\sum_{i=1}^m 1/{{\delta}^{(\kappa)}_{n_i}}}.
$
When no observations have yet been collected, the procedure randomizes according to the criterion based on the prior distribution alone, $\hat{\delta}^{(\kappa)}_{\beta_j}$, $j=1,\ldots,m$. Then, given $n_j$ observations, $\bx_j$ outcomes for arm $A_j$, $j=1,\ldots,m$ and using the `plug-in` estimator~(\ref{plugin}), arm $A_j$ is selected with probability $\hat{w}_j = 1$ if   $\hat{\delta}^{(\kappa)}_{n_j} = 0$ and with probability
\begin{equation} \hat{w}_j =  \frac{1/{\hat{\delta}^{(\kappa)}_{n_j}}}{\sum_{i=1}^m 1/{\hat{\delta}^{(\kappa)}_{n_i}}} \ {\rm if} \ \hat{\delta}^{(\kappa)}_{n_i} > 0, \ i=1, \ldots, m .
\label{weight00}
\end{equation}
The method proceeds until $N$ observations are attained. The arm $A_j$ satisfying
\begin{equation}
\hat{\delta}^{(1/2)}_{N_j} =\inf_{i=1,\ldots,m}\hat{\delta}^{(0.5)}_{N_i}.
\label{final}
\end{equation}
is adopted for the final recommendation, where $N_i$ is a total number of observation on an arm $A_i$, $i=1,2,\ldots,m$. The value $\kappa=0.5$ in (\ref{final}) is used  so that the final recommendation is not penalized by the sample size.

\subsubsection{Rule II: Select the best\label{sec:rule2}} 
Let $N$ be a total sample size and begin with the experiment with the arm that minimizes $\hat{\delta}^{(\kappa)}_{\beta_j}$, $j=1,\ldots,m$.  Given $n_j$ observations, $\bx_j$ outcomes for the arm $A_j$, $j=1,\ldots,m$ and using the `plug-in` estimator, an arm $A_j$ is selected if it satisfies
$\hat{\delta}^{(\kappa)}_{n_j} = \inf_{i=1,\ldots,m} \hat{\delta}^{(\kappa)}_{n_i}.$
The method proceeds until the total number of $N$ is attained. Again, we adopt $A_j$ as in (\ref{final}) for the final recommendation.

\subsection{Criterion in the context of clinical trials}
Further in the examples (Section \ref{sec:examples}) we apply the novel selection criterion to Phase I and Phase II clinical trials. In this case the arms are the different treatments (doses, combinations, schedules,...) and the goal is to find the treatment corresponding to specific toxicity (efficacy) characteristics. As the proposed information gain and corresponding selection criterion tends to assign the next patients to the best estimated TA during the trial, the criterion based on $\Delta_n$ is a \textit{patient's gain} (also known as \textit{`best intention`}) criterion as classified by \cite{whitehead1998}. The balance in the `exploration vs exploitation` trade-off is tuned by the term $n_j^{2 \kappa -1}$ reflecting the penalty on the number of observations on the same arm (i.e. many observations on one arm would favour selection of other arms). This implies that the design will keep selecting a specific arm only, if the corresponding estimate, $\boldsymbol{\alpha}$ is close $\boldsymbol{\gamma}$. As the trial progresses the design requires an increasing level of confidence that the selected arm is the TA. Clearly, $\kappa=1/2$ corresponds to no penalty and is of particular interest in trials with small sample sizes while larger values of $\kappa>1/2$ correspond to a greater interest in the statistical power of the experiment. Importantly, the first term in Equation (\ref{general_measure}) guarantees that the vast majority of patients will be assigned to the TA even for $\kappa>1/2$.

As many Phase I and Phase II clinical trials consider a binary endpoint $d=2$ (toxicity:yes/no or response:yes/no), we focus on this case in the examples. Then, (\ref{pdf}) reduces to the Beta-distribution and the proposed criterion takes the form
\begin{equation}
\hat{\delta}^{(\kappa)}_{n_j} (\alpha,\gamma) = \frac{1}{2} \frac{(\hat{p}_{n_j}^{(i)}-\gamma)^2}{\hat{p}_{n_j}^{(i)} (1-\hat{p}_{n_j}^{(i)})}n_j^{2 \kappa -1}
\label{measure}
\end{equation}
which is a normalized distance between $\hat{p}_{n_j}^{(i)}$ and $\gamma$. Note that this is not equivalent to the Euclidean distance $(\hat{p}_{n_j}^{(i)}-\gamma)^2$ which is a commonly used criterion for selection \citep[see e.g.][]{shen1996}. As the criterion tends to assign patients to the TA, it still has the Euclidean distance term in the nominator. However, it also takes into account the uncertainty in the denominator which is a variance of the probability of a binary event.
% Therefore, the criterion accounts for "how close the probability is to the target", but also for the level of uncertainty in the estimates.
The denominator can be also considered as a penalty term which "drives away" the allocation from the bounds ($\hat{p}_{n_j}^{(i)}=0$ or $\hat{p}_{n_j}^{(i)}=1$) as the boundary values correspond to the infinite value of the criterion. Note that, as the maximum variance of the binary probability is achieved at $\hat{p}_{n_j}^{(i)}=0.5$, so the criterion favours greater values of $\hat{p}_{n_j}^{(i)}$ which can be unethical if the objective of a study is to control the risk of toxicity. We will study whether the construction of the criterion creates any practical limitation in the context of Phase I clinical trial in Section \ref{phase1}.  

\section{Asymptotic behaviour\label{sec:asymptotics}}
Considering the asymptotic behaviour of a procedure ensures that the experimental design becomes more accurate as a sample size grows \citep{azriel2011}. Recall that the goal of the sequential experiment is to find an arm $j$ which corresponds to the minimum value ${\delta}_j \equiv \delta(\boldsymbol{\alpha}_j,\boldsymbol{\gamma})$ among all arms using random variables  $\tilde{\delta}^{(\kappa)}_{i}=\delta^{(\kappa)}(\bZ_{n_i},\boldsymbol{\gamma}),\ i=1,\ldots,m$.  Denote the arm to be selected by $\nu=\arg \min_{i=1,\ldots,m} \tilde{\delta}^{(\kappa)}_{i}$. Below, we consider the risk-adjusted average approach \citep{polley2008} and the probability of correct selection (PCS) \citep{cheung2013}
%\begin{equation}
$A_N= \frac{1}{m} \sum_{j=1}^{m} \mathbb{P}_{\pi_j}(\nu=j)$
%\label{accuracy}
%\end{equation}
where $\mathbb{P}_{\pi_j}$ is the probability computed under the vector $\pi_j = [\alpha_{1,j},\ldots, \alpha_{m,j}]^{T}$, $1 \leq j \leq m$, which assumes that $\alpha_{j,j}$ corresponds to the TA. The design is consistent if $\lim_{N \to \infty} A_N=\infty$.
The main result of this section is formulated in the following Theorem.
\begin{theorem}
Let us consider the experimental design with a selection criteria based on $\tilde{\delta}_{n_i}^{(\kappa)}$, $m$ arms and corresponding true probabilities vectors $\boldsymbol{\alpha}_i$, $i=1,\ldots,m$. Then, \\
\noindent \textbf{(a)} The design is consistent under Rule I (\ref{sec:rule1}) for $\kappa \geq 0.5$.

\noindent \textbf{(b)} The design is consistent under Rule II (\ref{sec:rule2}) for $\kappa > 0.5$.
\label{consistency_th}
\end{theorem}
\begin{proof}
The proof is given in the Appendix.
\end{proof}

\section{Examples\label{sec:examples}}
In this section the performance of the proposed experimental design is studied in the context of Phase I and Phase II clinical trials under different assignment rules. We will refer to our proposal as the Weighted Entropy (WE) design (WE${\rm _I}$ under Rule I and WE${\rm _{II}}$ under Rule II) and compare its performance to several well-established alternative approaches. All computations have been conducted using R \citep{Rcore}.

\subsection{Phase II clinical trial}
\subsubsection{Setting}
Consider a Phase II clinical trial whose primary endpoint is a binary measure of efficacy (e.g. response to treatment). The goals of the study are (i) to find the most effective treatment and (ii) to treat as many patients as possible on the optimal treatment. Clearly, Rule I is preferable for the first goal and Rule II for the second one. We consider two hypothetical trials, each with $m=4$ treatments, investigated by \cite{villar2015} for Multi-Arm Bandit models (MAB). We compare the performance of the proposed approach to the MAB approach based on the Gittins index \citep{gittins1979}, which is the optimal design in terms of maximising expected number of successes (ENS), and to fixed and equal randomization (FR) which is best in terms of the statistical power.

Trial 1 investigates $N_1=423$ and the true efficacy probabilities are $(0.3,0.3,0.3,0.5)$ while Trial 2 considers $N_2=80$ and the scenario $(0.3,0.4,0.5,0.6)$. Following the original application we consider the hypothesis $H_0: \ p_0 \geq p_i$ for $i=1,2,3$ with the family-wise error rate calculated at $p_0=\ldots=p_3=0.3$, where $p_0$ corresponds to the control treatment efficacy probability.  The Dunnett test \citep{dunnett1984} is used for the hypothesis testing in the FR setting. The hypothesis testing for MAB and WE design is performed using an adjusted Fisher's exact test \citep{agresti1992}. The adjustment chooses the cutoff values to achieve the same type-I error as the FR. The Bonferroni correction is used for MAB and WE designs to correct for the multiple testing and the family-wise error rate is set to be less or equal to 5\%.  Characteristics of interest are (i) the type-I error rate ($\alpha)$, (ii) statistical power $(1-\eta)$, (iii) the expected number of successes (ENS) and (iv) the average proportion of patients on the optimal treatment ($p^*$).

The WE design requires specification of the target value which can be specified in many clinical trial by clinicians. This can be defined as the maximum efficacy that they expect to see for the particular diseases. Below we consider the most challenging setting in which no target value is specified and the goal is to simply maximise the number of successes (as in MAB). To achieve this we set the target probability to a value close to 1. The target $\gamma=0.999$ is used which corresponds to the aim "to find the arm with the highest efficacy probability". The vector of the prior mode probabilities $p^{(0)}=[0.99,0.99,0.99,0.99]^{\rm T}$ is chosen to reflect no prior knowledge about which arm has the highest success probability and that each treatment is considered as highly efficacious until data suggests otherwise. This choice of prior  reflects the equipoise principle \citep{djulbegovic2000uncertainty}. We choose $\beta_0=5$ to ensure enough observations on the control arm and $\beta_1=\beta_2=\beta_3=2$ to reflect no prior knowledge for competing arms. We fix $\kappa=0.5$ for WE${\rm _I}$ and use different values of $\kappa$ for WE${\rm _{II}}$.

\subsubsection{Results}
The trade-off between the expected number of successes (ENS) and the statistical power for different values of the penalty parameter $\kappa$ under Rule II is illustrated in Figure \ref{fig:mab}.
\begin{figure}[!h]
  \centering
 \includegraphics[width=1\textwidth]{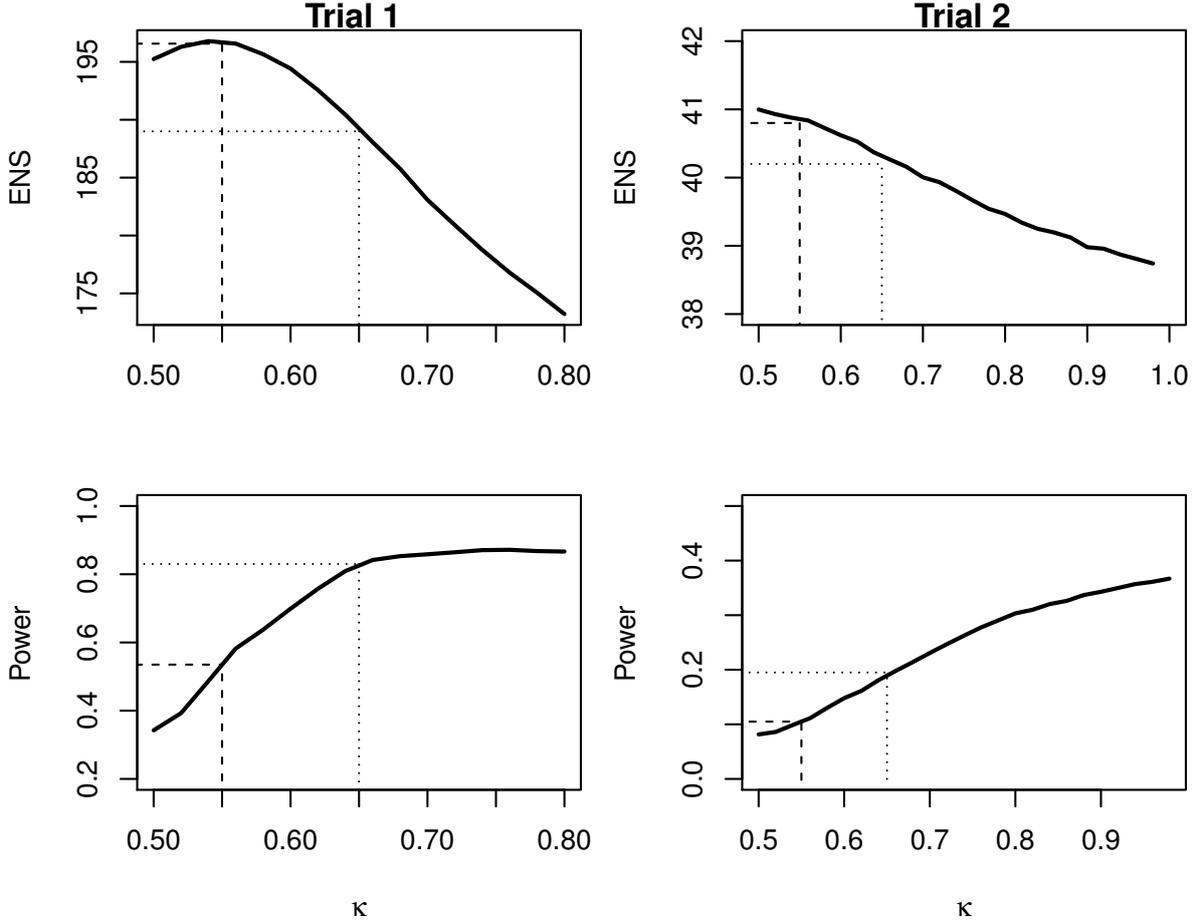}
      \caption{\small{ENS and power (fixed cutoff value) for the WE design under the Rule II for different $\kappa$. Dashed lines correspond to the values chosen for the subsequent study. Results are based on $10^4$ replications}}
      \label{fig:mab}
\end{figure}
In both trials, greater values of $\kappa$ correspond to greater power and lower ENS as the increase in penalty tends to more diverse allocations. The exception is $\kappa \in (0.5,0.55)$ in Trial 1 where the inconsistency for $\kappa=0.5$  leads to locking-in on the suboptimal treatment. We choose two values of $\kappa$ for the subsequent comparison. These choices correspond to (i) high ENS, but unacceptable power (dashed line) and (ii) slightly reduced ENS, but higher power (dotted line). 

The operating characteristics of considered designs in Trial 1 are given in Table \ref{tab:trial1}.
 \begin{table}[ht!]
  \centering
    \caption{\label{tab:trial1} Operating characteristics of the WE design under the Rule I ({\rm WE}${\rm _{I}}$), under the Rule II ({\rm WE}${\rm _{II}}$) for different $\kappa$ (in brackets), {\rm MAB} design and {\rm FR} in Trial 1 with $N=423$ under the null and alternative hypothesises. Results are based on $10^4$ replication.}
    \begin{tabular}{cccccccc}
 \multirow{2}{*}{Method} & \multicolumn{3}{@{}c}{$H_0:p_0=p_1=p_2=p_3=0.3$ } & & \multicolumn{3}{@{}c}{ $H_1:p_0=p_1=p_2=0.3, p_3=0.5$}   \\
 \cline{2-4}  \cline{6-8}  
 & $\alpha$ & $p^*(s.e)$ & ENS(s.e.) & & $(1-\eta)$ & $p^*(s.e.)$ & ENS (s.e.) \\
 \hline 
  MAB & 0.05 & 0.25 (0.18) & 126.68 (9.4) & & 0.43 & 0.83 (0.10) & 198.25 (13.7) \\
 FR & 0.05 & 0.25 (0.02)  & 126.91 (9.4) & & 0.82 & 0.25 (0.02) & 147.91 (9.6) \\
   WE${\rm _{I}}$ ($0.50$) & 0.05 & 0.24 (0.05)  & 126.84 (9.5) & & 0.88  & 0.39 (0.06) & 159.90 (11.0) \\
  WE${\rm _{II}}$ ($0.55$) & 0.05 & 0.21 (0.20) & 126.89 (9.4) & &0.55 & 0.83 (0.18)  &  197.13 (17.8) \\
   WE${\rm _{II}}$ ($0.65$) & 0.05  & 0.23 (0.13) & 126.86 (9.4) &  &0.87& 0.74 (0.10) &  189.26 (13.7) \\
 %  \hline
 %   \textbf{FR (Normal)} & 0.043 & 0.250 (0.02)  & 126.91 (9.4) & & 0.814 & 0.250 (0.02) & 147.91 (9.64) \\
 %  WE${\rm _{II}}$ ($\kappa=0.65$) & 0.044 & 0.226 (0.13) & 126.86 (9.4) &  &  0.820 & 0.738 (0.10) &  189.26 (13.7) \\
      \end{tabular}
      \end{table}
Under the null hypothesis, the performance of all methods is similar and the type-I error is controlled. Under the alternative hypotheses, the WE${\rm _{II}}$ design with $\kappa=0.55$ performs comparably to the MAB in terms of the ENS, but yields almost 10\% points increase in power. Nevertheless, it has unacceptable low statistical power which can be increased by using higher values of the penalty parameter ($\kappa=0.65$). It leads to an increase in the power from $0.53$ to $0.86$ at the cost of the slight ($\approx 4\%$) decrease in the ENS. In fact WE${\rm _{II}}$ then has comparable power to the FR, while treating almost 40 more patients on the superior treatment. Another way to increase the statistical power is to use WE${\rm _{I}}$ for which both the associated power and the ENS is higher than for the FR.

The operating characteristics of the designs for Trial 2 with fewer patients and a linear increasing trend is given in the Table \ref{tab:trial2}. 
      \begin{table}[ht!]
  \centering
    \caption{ \label{tab:trial2} Operating characteristics of the WE design under the Rule I (WE$_{\rm{I}}$), under the Rule II (WE$_{\rm{II}}$) for different $\kappa$ (in brackets), MAB design and FR in Trial 2 with $N=80$ under the null and alternative hypothesises. Results are based on $10^4$ replication.}
  \begin{tabular}{cccccccc}
 \multirow{2}{*}{Method} & \multicolumn{3}{@{}c}{$H_0:p_0=p_1=p_2=p_3=0.3$ } & & \multicolumn{3}{@{}c}{ $H_1:p_i=0.3+0.1i,i=0,1,2,3$}   \\
 \cline{2-4}  \cline{6-8}  
 & $\alpha$ & $p^*(s.e)$ & ENS(s.e.) & & $(1-\eta)$ & $p^*(s.e.)$ & ENS (s.e.) \\
 \hline 
  MAB & 0.00 & 0.25 (0.13) & 23.97 (4.10) & & 0.01 & 0.49  (0.21) & 41.60 (5.4) \\
   FR  & 0.05 & 0.25 (0.04) & 24.02 (4.10) & & 0.50 & 0.25 (0.04) & 35.98 (4.3) \\
   WE${\rm _{I}}$ ($0.50$) & 0.05  & 0.23 (0.07) & 23.92 (4.11) & & 0.59  &  0.33 (0.10) & 37.55 (4.7)  \\
   WE${\rm _{II}}$ ($0.55$) & 0.01 & 0.20 (0.15)  & 24.01 (4.10) & & 0.11 & 0.50 (0.27) & 40.72 (5.9)  \\
      WE${\rm _{II}}$ ($0.65$) & 0.05 & 0.22 (0.12)  & 23.96 (4.08) & & 0.52 & 0.47 (0.21) & 40.19 (5.4)  \\
 %  WE${\rm _{II}}$ ($\kappa=0.99$) & 0.049  & 0.239 (0.07)  & 23.97 (4.05) & & 0.593 & 0.376 (0.10) & 38.66 (4.8) \\
 %  \hline
  %     \textbf{FR (Fisher)} & 0.019 & 0.251 (0.04) & 24.02 (4.10) & & 0.368 & 0.250 (0.04) & 35.98 (4.3) \\
   %         WE${\rm _{II}}$ ($\kappa=0.99$) & 0.016  & 0.239 (0.07)  & 23.97 (4.05) & & 0.373 & 0.376 (0.10) & 38.66 (4.8) \\   
      \end{tabular}
      \end{table}
Under the null hypothesis, all designs perform similarly and type-I errors are controlled at the 5\% level. Under the alternative hypothesis, the MAB and WE${\rm _{II}}$ with $\kappa=0.55$, again, yield the highest (and similar) ENS among all alternatives, but also low statistical power. The WE${\rm _{I}}$ or increased $\kappa$ for WE${\rm _{II}}$ result  in a considerable power increase. Both designs have a greater (or similar) power and result in more ENS than the FR.

Overall, WE designs can perform comparably to the optimal MAB design in terms of the ENS, but with greater statistical power for both large and small sample sizes. They have similar statistical power to the FR, but with the considerably greater ENS. The ENS and power trade-off can be tuned via the built-in parameter $\kappa$. Although, some modification to the MAB designs were proposed \cite[e.g. see ][]{villar2015} to prevent the low statistical power, the majority of those are ruled-based. The proposed approach allows to avoid any algorithm-based rules and keeps the procedure fully adaptive. Additionally, the computation of the Gittens index for the MAB design is not trivial, requires special attention and is widely discussed in the literature \citep[e.g. see][and reference there in]{villar2015}. Some of them require calibration and can be computationally intensive. In contrast, the proposed criterion is extremely simple and easy to compute. While the proposed designs are compared for the target $\gamma=0.999$, similar performance is obtained for the problem of seeking an arm associated with a given response probability ($\gamma \in (0.7,1)$). This is, for example, of interest when seeking the effective dose 80 (ED80), the dose for which 80\% of subjects respond to the treatment. 

\subsection{Phase I clinical trial \label{phase1}}
\subsubsection{Setting}
To study the WE design in the context of Phase I clinical trials, let us consider $m=7$ arms, $N=20$ patients and the arm selection allowed after each patient. The goal is to find the arm (which could be combination, schedule or combination-schedule) with a toxicity probability closest to $\gamma=0.25$. In these studies randomization to all arms is not ethical for safety reasons and therefore, Rule II and $\kappa=0.5$ are used.
% Denote the true toxicity probabilities corresponding to arm $A_1, \ldots,A_7$ by $\alpha_1,\ldots,\alpha_7$.
 We would like to emphasize that we do not consider the classic dose-escalation problem in which the doses can be put according to increasing toxicity and focus on the setting in which clinicians cannot put arms according to increasing toxicity (as e.g. often in scheduling trials). While clinicians will be able to provide a presumed ordering of the arms, this order might be misspecified. 

We consider scenarios in which the prior order chosen by clinicians is either correct or misspecified. The scenarios with correctly specified ordering have a monotonic arm-toxicity relationship and the scenarios with misspecified ordering have a non-monotonic relationship. The investigated scenarios are shown in Figure \ref{shapes} and include a variety of monotonic and non-monotonic shapes as well as one setting with highly toxic arms only.
\begin{figure}[!ht]
  \centering
\makebox{ \includegraphics[width=1\textwidth]{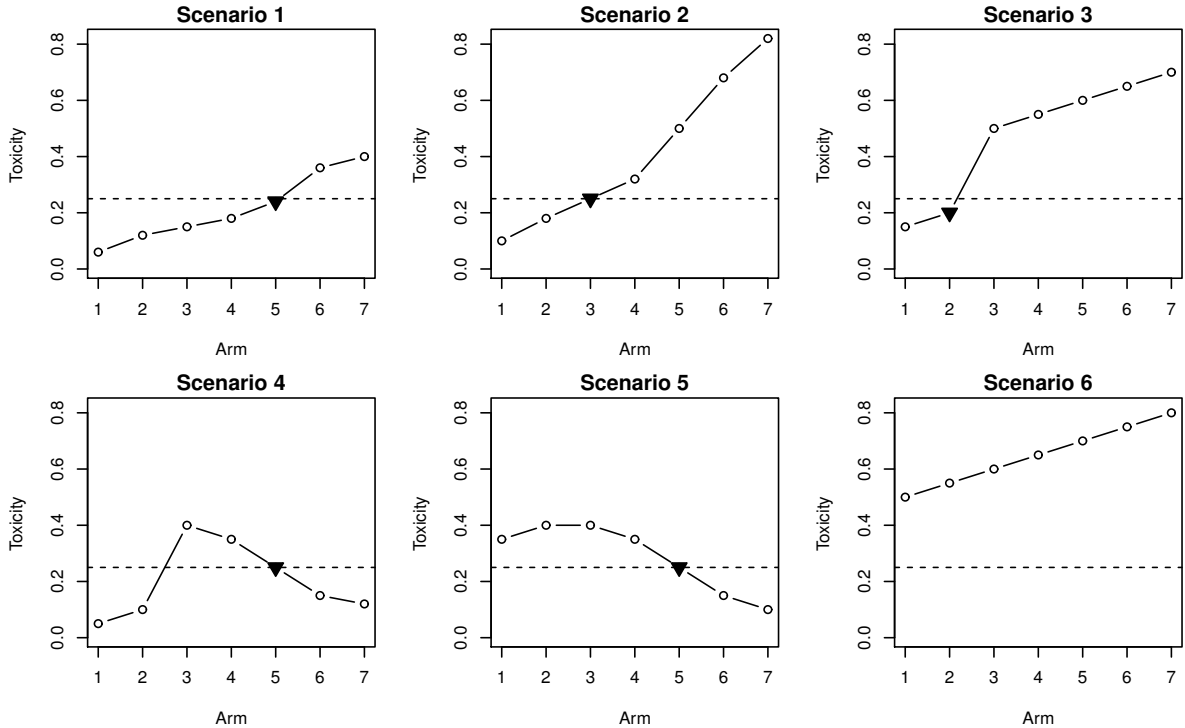}}
      \caption{Toxicity scenarios. The optimal arm is marked by a triangle and the maximum tolerated toxicity $\gamma=0.25$ is marked by dashed horizontal line. The monotonic scenarios (1-3, 6) correspond to correctly prespecified ordering of arms according to increasing toxicity and non-monotonic scenarios (4-5) to misspecified ordering of arms. \label{shapes}}
\end{figure}

It is assumed that limited information about treatments is available and a linear increase in the toxicity probabilities is expected such that $\hat{p}_{\beta_1}<\hat{p}_{\beta_2}<\ldots<\hat{p}_{\beta_7}$. For safety reasons, the trial is required to start at $d_1$. An `operational` prior, that is a prior that gives good operating characteristics under different scenarios, is calibrated. Details on the calibration are given in the Appendix. The resulting prior uses $\beta=1$ and prior toxicity risk modes of $\hat{\textbf{p}} = [0.25,0.3,0.35,0.4,0.45,0.5,0.55]^{\rm T}$. For the simulation study below, the penalty parameter is fixed at $\kappa=0.5$ due to the small sample size. 
 
The WE design is compared to common Phase I designs. Specifically the Bayesian Continual Reassessment Method (CRM) \citep{QPF90} and Escalation with Overdose Control (EWOC) \citep{EWOC}  are considered as methods that assume a monotonic toxicity relationship. Additionally, the partial ordering CRM  \citep[POCRM,][]{wages2011} which relaxes the monotonicity assumption is also considered. It uses the original CRM design with additional randomization among a pre-specified set of orderings. In our comparison we consider only correct orderings (Table \ref{tab:orders}) to allow for the best possible performance of the POCRM under the evaluated scenarios.
    \begin{table}[!ht]
 \caption{\label{tab:orders}Orderings for POCRM.}
 \centering
  \begin{tabular}{cc}
 & Order \\
\hline
1 & (1,2,3,4,5,6,7) \\
2 & (1,2,7,6,5,4,3) \\
3 & (7,6,5,4,1,2,3) \\
\hline
\end{tabular}
\end{table}
The same prior toxicity probabilities $\hat{\textbf{p}}$  and a rough prior distribution of the model parameters were chosen for the model-based alternatives. Finally, we include the non-parametric optimal benchmark \citep{benchmark} which provides the best theoretical performance if the patients' complete toxicity profiles are known.

The main characteristics to consider are: (i) the proportional of correct selections and (ii) the average number of toxic responses.  The \texttt{bcrm} package \citep{BCRM} is used for CRM and EWOC and the \texttt{pocrm} package is used for POCRM. For the proposed method and the non-parametric optimal design one-million-fold simulations are used while 100,000-fold simulations are used for the model-based methods due to computational constraints. 

\subsubsection{Safety constraint} 
For ethical reasons an escalation procedure should be planned so that only few patients are assigned to highly toxic treatments which is typically achieved by the use of a safety constraint. The majority of existing safety constraints are based on the assumption of monotonicity and hence are not suitable for the proposed design. We adopt the following safety constraint instead. The treatment $A_j$ is safe if after $n$ patients 
$
       \int_{\gamma^*}^{1}f_{n_{j}}(p) {\rm d}p \leq \theta_{n_j}
$
where $\gamma^*$ is an upper toxicity threshold, $\theta_{n_j}$ controls the overdosing probability and $f_{n_{j}}$ is the posterior Beta distribution for the toxicity probability. Note that the overdosing threshold $\theta_{n_j}$ changes as the trial progresses. It should be a decreasing function of $n$ with $\theta_0=1$ to give a possibility to test all the treatments (if data suggests so) and $\theta_{final} \leq 0.3$ to ensure that the final recommendation is safe. As an illustration, the linear non-increasing
$\theta_n=\max(1-rn,\theta_{final})$
is used with $r>0$. We have calibrated the parameters of the safety constraints (details in the Appendix) and used $\gamma^*=0.45$ and $r=0.035$ in the simulations. Similar safety constraints were incorporated in the model-based methods.

\subsubsection{Results}
The simulation results in monotonic scenarios 1-3 are given in Table \ref{safcon1}.
       \begin{table}[!ht]
\caption{The operating characteristics of the WE, CRM, POCRM and EWOC designs. `Term`, `T` and $\bar{N}$ correspond to the termination proportion, the average number of toxic responses and the average number of patients, respectively. The most likely recommendation is in bold, the actual target regimen is in italics. \label{safcon1}}
  \small
  \centering
%  \begin{tabular}{rrrrrrrrrrr}
   \begin{tabular} {p{1.25cm}WWWWWWWQQQ}
    \hline
     & $d_1$ & $d_2$ & $d_3$ & $d_4$ & $d_5$ & $d_6$ & $d_7$ & Term & T  & $\bar{N} $ \\ 
    \hline
    \multicolumn{11}{@{}c}{{Scenario 1. Linear response}}\\
    Scenario & 0.06 & 0.12 & 0.15 & 0.18 & \textit{0.24} & 0.36 & 0.40 & & &  \\ 
                Optimal & 0.92 & 9.12 & 10.60 & 14.44 & \textbf{31.54} & 27.50 & 5.87 &  \\ 
       WE& 7.03 & 14.72 & 23.33 & \textbf{30.11} & 23.34 & 1.39 & 0.05 & 0.1 & 3.36 & 20.0 \\ 

          CRM & 2.66 & 7.21 & 14.17 & 20.58 & \textit{\textbf{26.62}} & 15.95 & 12.53 & 0.3 & 4.38 & 19.9\\ 
           POCRM  &2.69 & 11.25 & 22.30 & 15.73 & \textit{\textbf{22.60}}  & 20.62  & 4.60 & 0.2 &  4.94 & 20.0 \\ 
           EWOC  &  8.06 & 13.80 & 20.30 & {\textbf{23.70}} & \textit{20.20} & 9.62 & 3.88 & 0.4 & 3.76 &  18.8\\ 
    \hline
     \multicolumn{11}{@{}c}{{Scenario 2. Logistic shape}}\\
        Scenario & 0.10 & 0.18 & \textit{0.25} & 0.32 & 0.50 & 0.68 & 0.82 & & & \\ 
                    Optimal & 6.05 & 29.03 & \textbf{30.12} & 28.27 & 6.48 & 0.05 & 0.00 & \\  
     WE & 16.78 & 26.43 & \textit{\textbf{29.54}}& 22.51 & 3.76 & 0.11 & 0.00 & 0.9 & 5.23 & 20.0 \\      
                   CRM  & 17.24 & 25.88 & \textit{\textbf{28.70}} & 19.37 & 6.24 & 0.56 & 0.04 & 1.9  & 4.84 & 19.7 \\ 
                    POCRM  &  14.98& 27.32 & \textit{\textbf{27.89}} & 18.50  & 6.70 & 1.04 & 1.86 & 1.7  & 5.54 & 20.0 \\ 
                     EWOC   & \textbf{28.72} & \textbf{27.66} & \textit{24.32} & 13.65 & 4.21 & 0.00 & 0.00 & 1.4 & 3.27  & 18.0 \\
               \hline
    \multicolumn{11}{@{}c}{{Scenario 3. J shape}}\\
              Scenario & 0.15 & \textit{0.20}  & 0.50 & 0.55 & 0.60 & 0.65 & 0.70 & & &  \\ 
%    Optimal & 42.40 & \textbf{45.79} & 11.49 & 0.25 & 0.06 & 0.01 & 0.00 &  \\ 
Optimal & 29.87	& 58.31& 	10.0	2 & 1.69& 	0.11& 	0.00	& 0.00& \\	
     WE & 38.07 & \textit{\textbf{44.65}} & 6.59 & 3.44 & 1.48 & 0.28 & 0.02 & 5.5 & 5.94 & 19.8  \\ 
              CRM & \textbf{37.47} &\textit{\textbf{37.85}} & 17.41 & 2.92 & 0.36 & 0.07& 0.00 & 3.9 & 5.10 & 19.4 \\ 
     POCRM  & \textbf{33.57} & \textit{\textbf{37.76}} & 13.27 & 2.55 & 0.54 & 1.33  & 6.04  & 4.9 & 6.06  & 19.8 \\ 
       EWOC  & \textbf{51.00} & \textit{26.11} & 11.01 & 0.88 & 0.13 & 0.00 & 0.00 & 10.9 & 3.60 & 16.8 \\ 
             \hline
  \end{tabular}
\end{table}
The WE design performs comparably to the CRM and POCRM designs and recommends the correct treatment with the probability nearly $0.25$ and $0.30$ in scenario 1 and 2. In scenario 1 the WE design underestimates the target treatment and recommends a less toxic treatment more often due to the safety constraint. Despite that, the performance of all methods is not far from the non-parametric optimal benchmark which shows that the detection of the target treatment is quite challenging. Proportions of terminations are close to $0$ and the average number of toxic responses is largely the same. For the EWOC, the level of the target treatment is underestimated in both scenarios. 

In scenario 3, the WE design shows a better performance than the model-based alternatives with nearly 45\% of correct recommendations against about 40\% for CRM and POCRM. The safety constraint allows to prevent the recommendation of highly toxic treatments and controls the total number of toxic responses. Again, the EWOC underestimates the target therapy, but results only in $3$ toxic responses compared to $5$ for the CRM and $6$ toxicities for the WE and the POCRM. As expected, methods that relax monotonic assumption result in more toxic responses that monotonicity based designs.

The results for non-monotonic (Scenarios 4-5) and unsafe (Scenario 6) cases are given in Table \ref{safcon2}.
 \begin{table}[!ht]
\caption{The operating characteristics of the WE, CRM, POCRM and EWOC designs. `Term`, `T` and $\bar{N}$ correspond to the termination proportion, the average number of toxic responses and the average number of patients, respectively. The most likely recommendation is in bold, the actual target regimen is in intalics. \label{safcon2}}
 \small
  \centering
   \begin{tabular} {p{1.15cm}WWWWWWWQQp{0.55cm}}
    \hline
     & $d_1$ & $d_2$ & $d_3$ & $d_4$ & $d_5$ & $d_6$ & $d_7$ & Term & T  & $\bar{N} $ \\ 
    \hline
                \multicolumn{11}{@{}c}{{Scenario 4. Inverted-U shape}}\\
         Scenario & 0.05 & 0.10 & 0.40 & 0.35 & \textit{0.25} & 0.15 & 0.12 &  & &\\  
             Optimal & 0.88 & 7.36 & 19.12 & 18.96 & \textbf{38.47} & 13.64 & 1.57 & \\ 
          WE & 14.11  &19.13 & 11.77 & 18.27 & \textit{\textbf{27.90}} & 8.50 & 0.23 & 0.1& 4.26 & 20.0 \\      %     WDE$_{\rm SC}$ (Bayesian) & 14.11  &19.13 & 11.77 & 18.27 & \textit{\textbf{27.90}} & 8.50 & 0.23 & 0.15& 4.26 & 19.99 \\ 
             CRM & 4.26 & \textbf{19.90} & 17.70  & 6.31  &\textit{ 2.84}  & 3.00  & \textbf{46.10}  & 0.3 & 3.26  & 19.9 \\ 
             POCRM  & 2.87 & 11.39  & 11.75  & 9.32  & \textit{19.11}  & \textbf{33.94} & 11.62 & 0.2  &  4.29  & 20.0 \\ 
           EWOC & 7.18 & 24.90 & 18.60  & 3.79  &\textit{ 2.52}  & 3.79  & 30.60  & 6.6& 2.73 & 18.9 \\ 
                     \hline
               \multicolumn{11}{@{}c}{{Scenario 5. Inverted-U shape}}\\
                              Scenario & 0.35 & 0.40 & 0.40 & 0.35 & \textit{0.25} & 0.15 & 0.10 &  & & \\ 
                              Optimal  & 16.18 & 3.01 & 3.01 & 16.18 & \textbf{39.46} & 18.65 & 3.51 &  \\ 

          WE& 15.57 &12.65 & 13.31 & 18.27  & \textit{\textbf{27.92}}  & 8.90  &0.58  &9.9 & 5.81  & 19.7 \\     
             CRM & \textbf{47.41} & 2.51 & 0.97 & 0.48  & \textit{0.72}  & 0.40  & \textbf{30.10}  & 27.3 & 4.27 & 16.0 \\ 
              POCRM & 16.81 & 5.98  & 5.66  & 12.42  & \textit{20.10} & \textbf{23.13} & 10.23 & 9.7 &  5.14  & 19.5 \\ 
           EWOC  & \textbf{30.75} & 1.26 & 0.78  & {0.47}  &\textit{0.47 } & 0.31  & 9.78 & \textbf{56.2} & 3.30 &  11.0 \\ 
           \hline
       \multicolumn{11}{@{}c}{{Scenario 6. Unsafe}}\\
         Scenario & 0.50 & 0.55 & 0.60 & 0.65 & 0.70 & 0.75 & 0.80 & & & \\ 
             Optimal  & 80.53 & 16.35 & 3.10 & 0.02 & 0.00 & 0.00 & 0.00 &  \\
            WE& 13.63 & 5.53 & 2.45 & 0.88 & 0.27 & 0.06 & 0.00 & \textit{\textbf{77.2}}& 8.02 & 14.2\\ 
          CRM & 32.24 & 0.32 & 0.08 & 0.00 & 0.00 & 0.00 & 0.00 & \textit{\textbf{67.4}}& 5.33  & 10.3 \\ 
         POCRM  & 13.18 & 0.57 & 0.12 & 0.04  & 0.01  & 2.06  & 0.08 & \textit{\textbf{83.9}} &  7.12 & 12.5 \\ 
          EWOC  & 16.17 & 0.00 & 0.12 & 0.00 & 0.00 & 0.00 & 0.00 & \textit{\textbf{83.7}} & 3.07 & 6.1 \\  
          \hline
  \end{tabular}
\end{table} As expected, the designs based on the monotonicity assumption are not able to find the target treatment in non-monotonic settings. Comparing other designs, the WE design has a substantial advantage. It finds the correct arm with the probability nearly $0.28$ compared to $0.20$ for the POCRM while exposing nearly the same number of patients to toxic treatments. The safety constraint allows recommendation of the optimal treatment even in non-monotonic scenarios where the target treatment lies beyond the toxic treatments ($d_3-d_4$).

Considering Scenario 6, the WE design terminates earlier with probability $0.8$ and performs similar to the POCRM and EWOC. It outperforms the CRM which recommends a highly toxic arm with a larger probability ($32.56\%$ against $19.16\%$).  However, methods that relax the monotonicity assumption result in more toxic responses and require more patients on average to come to the termination conclusion.  While the CRM and the EWOC require $5$ and $6$ patients only, it takes nearly $14$ and $13$ patients for the WE and POCRM as they explore all arms before concluding none is safe. This, however, would not be considered as a severe drawbacks of these methods as in many application clinicians would stop the trial for safety reasons based on extensive patients profile.

Summarizing, the proposed design performs comparably to the model-based approaches in monotonic settings and clearly outperforms them in non-monotonic ones. Importantly, despite the denominator of the proposed criterion which might favour slightly more toxic arms, the design does not cause any practical concerns as soon as the allocation is restricted to the safety set. In other words, the criterion allocates the patients to the best estimated TA taking into account information about the uncertainty in the estimates and about the safety set. Importantly, the time-varying safety constraint achieves the goals motivated by the ethical concerns while not preventing the target treatment selection in safe and non-monotonic scenarios. One can conclude that the design is ethical and can be applied in practice.

 \section{Discussion\label{sec:discussion}}
In this work, we propose a family of criteria for selecting the best arm in experiments with multinomial outcomes. The novel criterion leads to accurate selection without the need for parametric or monotonicity assumptions. The fundamental property of the criterion is the infinite penalization of the bounds which was argued to be a crucial property for a parameter defined on the restricted space \citep{ait}. This property drives the allocation away from the bounds of a restricted space to the neighbourhood of the target value. The consistency conditions of the proposed design and exact rate for the special case of binary outcomes are obtained. It is shown how one can benefit from the proposed design in Phase I and Phase II clinical trials. The proposal was demonstrated to have a comparable or better characteristics than other alternatives. It preserves flexibility and allows to tailor the design parameters in light of the investigation goal. Additionally, the design is computationally simple and a large set of simulation can be performed in a feasible time.

Despite clinical trials being used as the main motivation throughout, the design can be applied to a wide range of problems of a similar nature. For example, applications where the MAB approach has found application: online advertising, portfolio design, queuing and communication networks, etc. \citep[see][and references there in]{gittinsbook}. On top of that, the proposed design can be used in more general problems of percentile estimation rather than the identification of the highest success probability. It is important to emphasize that the derived selection criterion can be also applied in conjunction with a parametric models which also expands its possible applications. In fact, the parameters can be estimated by any desirable method and then `pluged-in` in the criterion which preserves its properties.

\section*{Acknowledgments}

This project has received funding from the European Union's Horizon 2020 research and innovation programme under the Marie Sklodowska-Curie grant agreement No 633567 and, in part, from from Prof Jaki's Senior Research Fellowship (NIHR-SRF-2015-08-001) supported by the National Institute for Health Research. The views expressed in this publication are those of the authors and not necessarily those of the NHS, the National Institute for Health Research or the Department of Health.

\bibliographystyle{biorefs}
\bibliography{mybib}

\clearpage
\begin{appendices}
\section{Proofs}
\subsection*{Proof of Theorem 2.1}
\begin{proof}
The problem reduces to computing integrals of the following forms
$$
\int_{\mathbb{S}^d} \log \left(p^{(i)}\right)^{x^{(i)}}  f_n(\bp) {\rm d} \bp = x^{(i)} \left( \psi \left(x^{(i)}+1 \right) -\psi \left( n+d \right) \right), \ i=1,2,\ldots,d,
$$
$$
\int_{\mathbb{S}^d} \log  \left(p^{(i)}\right)^{x^{(i)}} \phi_n(\bp) f_n(\bp){\rm d} \bp = x^{(i)} \left( \psi \left(x^{(i)} + \gamma^{(i)} n^{\kappa}+1 \right) -\psi \left( n+ n^{\kappa} + d \right) \right), \ i=1,2,\ldots,d
$$
where $\psi(x) = \frac{{\rm d}}{{\rm d}x} {\rm log} 
\Gamma(x)$ is the digamma function. Using the asymptotics of the digamma function \citep{rg2007}, Taylor series expansion and simplifying terms, the results immediately follows.
\end{proof}
\subsection*{Proof of Theorem 2.2}
\begin{proof}
Following the proof of Theorem 2.1, one can obtain that
$h\left( f_n \right) = \log B \left( \bx+\bv + \bJ  \right)  + n \psi \left(n+d \right) - \sum_{i=1}^{d} x^{(i)} \psi \left(x^{(i)}+1 \right).$
As $n \to \infty$
$$ h\left( f_n \right) = \frac{1}{2}\log \left( 2 \pi e \right)^{d-1} + \frac{1}{2}\log \frac{\prod_{i=1}^d\alpha^{(i)}}{n^{d-1}} + O \left(\frac{1}{n} \right).$$ 
Using that
$ h\left( \widetilde{f}_n \right) =  h\left( f_n \right) + \log | det \left(\Sigma^{-1/2} \right)|$ \citep{cover2006} where $det(A)$ is a determinant of the matrix $A$ and 
$\log | det \left(\Sigma^{-1/2} \right)| = - \frac{1}{2} \log \frac{\prod_{i=1}^d\alpha^{(i)}}{n^{d-1}}+ O \left(\frac{1}{n}\right)$ as $n \to \infty$. Then,
\begin{eqnarray*}
 \mathbb{D}(\widetilde{f}_n\ ||\ \varphi)&=& - h\left(\widetilde{f}_n \right) - \int_{\mathbb{S}^d} \widetilde{f}_n(\bp) \log \varphi(\bp) {\rm d}\bp \\
&=& - \frac{1}{2}\log \left( 2 \pi e \right)^{d-1} + \frac{1}{2}\log \left( 2 \pi\right)^{d-1} + \frac{1}{2} \int_{\mathbb{S}^d} \sum_{i=1}^{d-1}\left(p^{(i)}\right)^2 \widetilde{f}_n(\bp) {\rm d} \bp + O \left( \frac{1}{n} \right)\\ &=&  O \left( \frac{1}{n} \right), 
\end{eqnarray*}
as  $\int_{\mathbb{S}^d} \sum_{i=1}^{d-1}\left(p^{(i)}\right)^2 \widetilde{f}_n(\bp) {\rm d} \bp = d-1 +O \left( \frac{1}{n} \right) $ is the sum of the second moments. Using Pinsker's inequality \citep{csiszar2011}  it implies the convergence in total variation \citep{cover2006} which implies the weak convergence. 
\end{proof}

\subsection*{Proof of Theorem 3.1}
\begin{proof}
\noindent \textbf{(a) Rule I.}
Under Rule I the proportion of observations on each arm converges to a constant (Theorem 2.3). Therefore, it is initially assumed that it is fixed and the probability measures below are conditional on the  allocation proportion.

We start from from $\kappa=1/2$ and adopt notation $\tilde{\delta}_{n_i}^{(1/2)}\equiv \tilde{\delta}_{i}$. Denoting 
$\bar{C}_{k,k+1} \equiv \{ \tilde{\delta}_{k} < \tilde{\delta}_{k+1} \} \ {\rm and} \ {C}_{k,k+1} \equiv \{ \tilde{\delta}_{k} > \tilde{\delta}_{k+1} \}$ we find that 
$$ \displaystyle \{ \nu = k \} \Leftrightarrow \{ \bar{C}_{k,1} \cap \bar{C}_{k,2} \cap \ldots \cap \bar{C}_{k,k-1} \cap  \bar{C}_{k,k+1}\cap \ldots \cap  \bar{C}_{k,m} \} \Leftrightarrow \{ \cap_{i=1,i \neq k}^{m} \bar{C}_{k,i} \}.$$

Using DeMorgan's law and Boole's inequality \citep{resnick2013}, one can obtain
\begin{equation}
\mathbb{P}(\nu = k)=1-\mathbb{P} \left( {\cup_{i=1,i \neq k}^{m}{C}_{k,i}} \right) \geq 1- \sum_{i=1, i \neq k}^{m} \mathbb{P} \left({C}_{k,i} \right)=1- \sum_{i=1, i \neq k}^{m} \mathbb{P} \left(\tilde{\delta}_{k} > \tilde{\delta}_{i} \right)
\label{prob}
\end{equation}
where
\begin{equation}
\mathbb{P}\left( \tilde{\delta}_{k} > \tilde{\delta}_{i} \right)=\mathbb{P}\left( \frac{  \tilde{\delta}_{i} - {\delta}_{i} -\tilde{\delta}_{k} + {\delta}_{k} }{\Sigma_{k,i}}  < \frac{{\delta}_{k}  - {\delta}_{i}}{\Sigma_{k,i}}  \right) \approx \Phi \left( \frac{{\delta}_{k}  - {\delta}_{i}}{\Sigma_{k,i}}  \right) 
\label{asymptotics}
\end{equation}
with $\Phi(.)$ denoting the distribution function of a standard normal random variable, $\Sigma_{k,i}~=~\left(\Sigma_k+\Sigma_{i}\right)^{1/2}$ and $\Sigma_k$ the variance corresponding to arm $k$ as in the Theorem 2.3. As arms $k$ and $i$ are independent, there are two independent random variables in the left-hand side of the second term in (\ref{asymptotics}) and each of them converges to a Gaussian random variables (Theorem 2.3). Therefore, the sum converges to a standard Gaussian random variable after an appropriate normalization. Consequently, for $j=1,\ldots,m$
\begin{equation}
 \mathbb{P}_{\pi_j} \left( \nu = j \right) \geq
1-\sum_{i=1, i \neq j}^{m}\Phi \left( \frac{\delta_{j,j} - \delta_{i,j} }{\Sigma_{j,i} }  \right)  
\label{PCS}
\end{equation}
By the construction of $\pi_j$, ${\delta}_{j,j} - {\delta}_{i,j}<0$. The number of observations on each arm $N_i$ is proportional to the total sample size $N$ under Rule I: $N_j \simeq w_jN$ and $N_j \simeq w_{j-1}N$  where $a_n \simeq b_n$ means that $\lim_{n \to \infty}\frac{a_n}{b_n}=1$. Thus, 
\begin{equation}
\frac{\delta_{m,m}  - \delta_{i,j} }{\Sigma_{m,i} } \simeq c \sqrt{N} 
\label{limit}
\end{equation}
where $c$ is a negative constant. Plugging-in terms in the accuracy formula, we obtain that
$
\lim_{N \to \infty} A_N \geq 1.
$

For $\kappa>1/2$, the probability of the final selection in the experiment is still given by (\ref{prob}) for $\kappa=1/2$ as the penalty term is not taken into account for the final recommendation. Then, the only difference is the number of observation on each arm, proportional to the total number of patients $N_j \simeq l_j(N) N$ with $l_j$ depending on $N$. This results in a different constant $c<0$ in (\ref{limit}), but in the unchanged rate $\sqrt{N}$ due to the same rate in both nominator and denominator in Equation (2.8) with respect to $N$.
\subsubsection*{Binary outcomes}
While the asymptotic result (\ref{asymptotics}) is given in terms of $\Sigma_{k,i}$, it can be written explicitly. In the special case of binary outcomes $d=2$, $\Sigma_{k,i} = \sqrt{\sigma_k^{2} + \sigma_{i}^2} $ and $
\sigma_j= |{{\delta}_{i}}^{\prime}| \sqrt{\frac{{\alpha_j(1-\alpha_j)}}{{n_j}}}$, $j=i,k$
with
$$\left.\frac{\partial{\delta}^{(\kappa)}_{i}\left(z,\gamma \right)}{\partial z}\right\vert_{z=\alpha} = {{\delta}^{(\kappa)}_{i}}^{\prime} =\frac{(\gamma-\alpha)(\gamma(2\alpha-1)-\alpha)}{\alpha^2(1-\alpha)^2}.$$
Therefore,
$$ \Phi \left( \frac{{\delta}_{k}  - {\delta}_{i}}{\sigma_{k,i}}  \right) = \Phi \left( \frac{\sqrt{n_{k}n_{i}} \left({\delta}_{k}  - {\delta}_{i} \right)}{\sqrt{n_{k} \alpha_i(1-\alpha_i)({{\delta}_{i}}^{\prime})^2+n_{i} \alpha_{k}(1-\alpha_{k})({{\delta}_{k}}^{\prime})^2}}  \right).$$
From the expression above, the rate obtained in (\ref{limit}) is explicit.\\

\noindent \textbf{(b) Rule II.}
Consider  $\kappa=\frac{1}{2}$. The design based on this measure and on its point estimate is inconsistent as it does not guarantee an infinite number of patients on all arms. We use an example with two arms only with the arm $A_1$ being the optimal. Suppose that prior parameters are specified such that $\hat{\delta}_{\beta_2} \ll \hat{\delta}_{\beta_1}$ and $\delta_2 \ll \hat{\delta}_{\beta_1}$, so $A_2$ is selected initially. While the number of observations on the arm 2 increases and the estimate $\hat{\delta}_2$ approaches the true value $\delta_2$ (Theorem 2.3), the estimate $\hat{\delta}_1$ remains unchanged. One can find prior values $\hat{\delta}_{\beta_1} \gg \delta_2$ such that $A_1$ is never  selected, because the point estimate $\hat{\delta}_2$ would not go below $\hat{\delta}_{\beta_1}$. Consequently, the selection would get stuck at the suboptimal arm regardless of further outcomes. So, the number of patients on both arms does not tend to infinity as $N \to \infty$.

For $\frac{1}{2} < \kappa <1$, let $N_j^k$ be the indicator function such that
$$N_j(t)= \begin{cases}
1 \ {\rm with \ probability} \ \mathbb{P}(\tilde{\delta}_{j}^{(\kappa)} (t)=\min_i \tilde{\delta}_{i}^{(\kappa)}(t)) \\
0 \ {\rm with \ probability} \ 1-\mathbb{P}(\tilde{\delta}_{j}^{(\kappa)} (t)=\min_i \tilde{\delta}_{i}^{(\kappa)}(t)) 
\end{cases}$$
where $\tilde{\delta}_{j}^{(\kappa)} (t)$ is a random variable corresponding to the posterior density function after $t$ observations in the experiment.
Let $n_j(t)=\sum_{i=1}^{t}N_j(i)$ be the number of observations on arm $j$ up to the moment $t$. We then obtain
$$\mathbb{E}(n_j(t)) = \sum_{u=1}^{t} \mathbb{E} N_j(u)=\sum_{u=1}^{t} \mathbb{P}(\tilde{\delta}_{j}^{(\kappa)} (u)=\min_i \tilde{\delta}_{i}^{(\kappa)}(u)).$$
Note, that $\mathbb{P}(\tilde{\delta}_{j}^{(\kappa)} (u)=\min_i \tilde{\delta}_{i}^{(\kappa)}(u))$ has already been studied in (a). The mean of $\tilde{\delta}_{j}^{(\kappa)} (u)$ associated with an arm $A_j$ to be selected is an increasing polynomial with respect to  $N$. The probability to be the minimum decreases for $j$ and increases for $i=1,\ldots,j-1,j+1,\ldots,m$.
It follows that the probability of being selected is not a monotonic function and
$$\lim_{t \to \infty} \sum_{u=1}^{t} \mathbb{P}(\tilde{\delta}_{j}^{(\kappa)} (u)=\min_i \tilde{\delta}_{i}^{(\kappa)}(u))=\infty.$$
The final selection is the arm satisfying (2.9). Consequently, the number of observations on each arm tends to infinity and we obtain that $\lim_{N \to \infty}A_N=1$ using the arguments of (a).
\end{proof}

\section{Parameters calibration for the Phase I clinical trial}
\label{app2}
\subsection{Operational prior}
A prior for treatment $d_j$  can be specified through the mode of the prior distribution, $\hat{p}_{\beta_j} = \frac{\nu_j}{\beta_j}$. To guarantee the procedure to start from $d_1$ we set $\hat{p}_{\beta_1}  = \gamma$. As an investigator has the same amount of knowledge about each treatment we set $\beta_1=\ldots=\beta_7=\beta$.  Larger values of $\beta$ and the rate of increase correspond to the more conservative escalation scheme as an investigator needs more observations on each particular treatment to escalate. Similarly, a greater differences in prior toxicity probabilities would correspond to the more conservative scheme as well, because it would require more evidence to escalate. Therefore, one can expect a set prior parameters that would lead to a similar PCS. The investigation of these parameter influence on the operational characteristics of the proposed approach is given below. 

We consider six different scenarios with different location of the target (Figure \ref{fig:oper_prior}). For simplicity, we would assume a priori that toxicity increases linearly between treatments. Given $m=7$, we set the difference between prior toxicities on the $d_1$ and $d_7$ and, then, interpolate the linear curve for the rest. We would define $step=\hat{p}_{\beta_7}-\hat{p}_{\beta_1}$. Then, we vary values of $step$ and $\beta$ for each scenarios. The PCS for different combination of $step$ and $\beta$ is given in the Figure \ref{fig:oper_prior}.

\begin{figure}[ht!]
  \centering
    \includegraphics[width=1\textwidth]{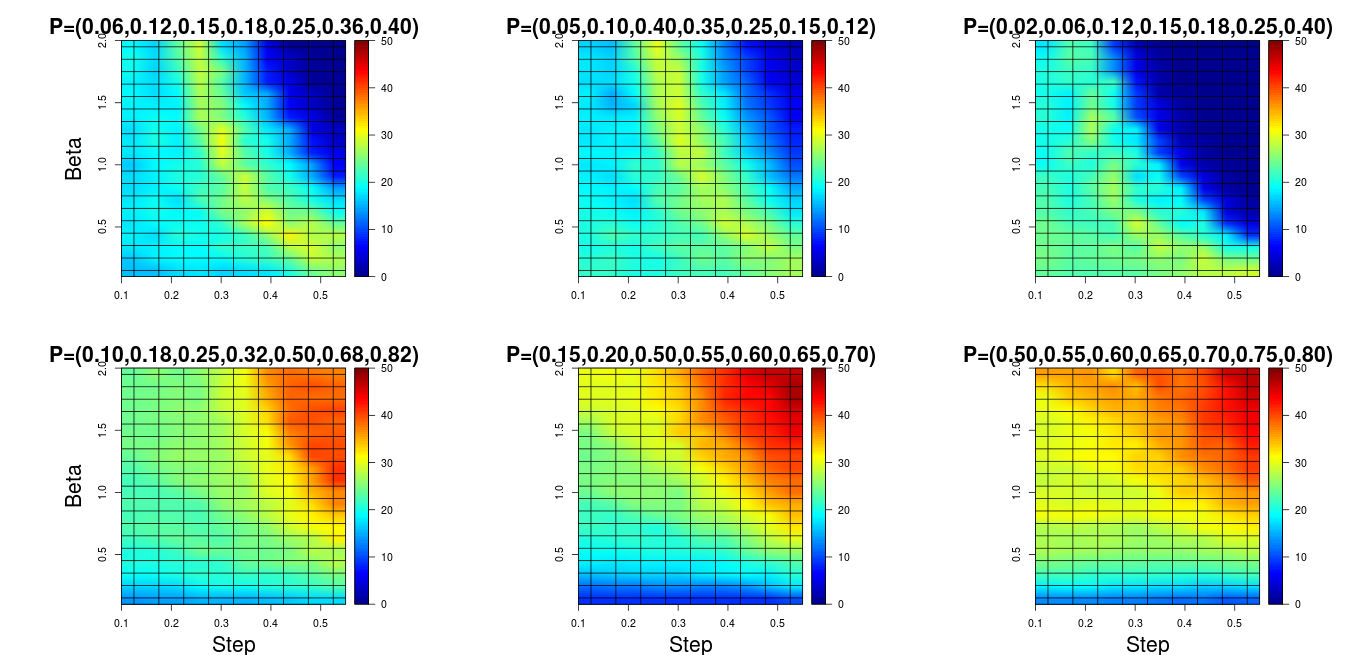}
      \caption{The PCS for the WE design using $N=20$ and different combinations of $\beta$ (vertical axis) and $step$ (horizontal axis). Results are based on $10^6$ simulations. \label{fig:oper_prior}}
\end{figure}

Brighter colours correspond to higher values of the PCS. A conservative prior (top right corner on the grid) prevents the WE design from the correct recommendation in upper line graphs scenarios as higher doses can be hardly reached with $N=20$. At the same time, it leads to an accurate selection in scenarios with highly toxic doses (lower line scenarios). In contrast, less conservative prior results in higher proportion of correct recommendation in upper line scenarios and worse in lower line ones. Therefore, there is a trade-off between the ability to investigate higher doses and the desire to prevent the high number of toxic responses. Therefore, the geometric mean of the PCS over all scenarios is chosen as the criterion for the operational prior choice. The geometric mean for different set of parameters is given in the Figure \ref{fig:geom_mean}.
 
 \begin{figure}[!ht]
\centering
\makebox{\includegraphics[width=0.6\textwidth]{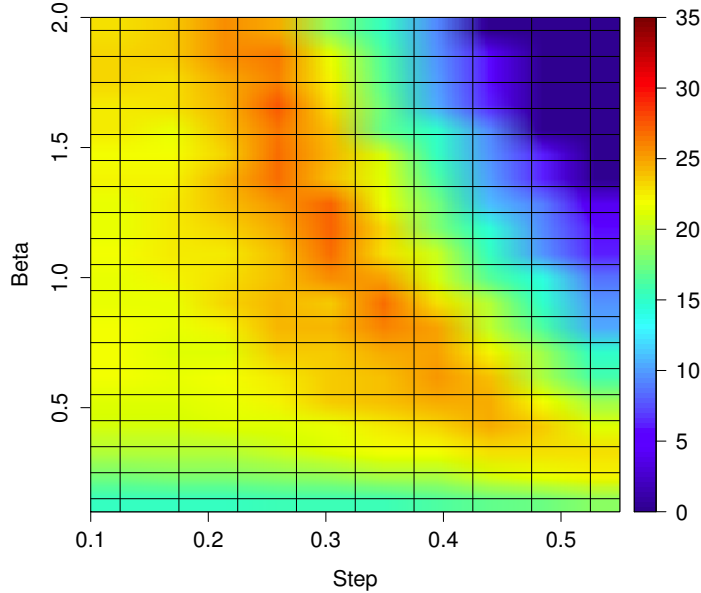}}
  \caption{\label{fig:geom_mean}The geometric mean of the proportion of correct recommendations by the proposed WE method using different set of prior parameters: $\beta$ (vertical axis) and the difference between the risk of toxicity on the lowest and highest dose (horizontal axis) in six scenarios: with the TD at the bottom, in the middle and at the top of the investigated dose range. $10^6$ simulations are used.}
\end{figure}

There is a set of the prior parameter that lead to the same geometric mean of the PCS across all scenarios. We choose a prior that carries limited information $\beta=1$ and fix the rate to maximize the geometric mean among all scenarios. Thus, the following vector of modes $\hat{\textbf{p}}$ is chosen $$\hat{\textbf{p}} = [0.25,0.3,0.35,0.4,0.45,0.5,0.55]^{\rm T}.$$ 

We also consider the case of $N=25$ thought the same set of scenarios to illustrate the influence of the prior parameter for a larger sample size (Figure \ref{n25}). There is a a similar dependence pattern on $\beta$ and $step$. However, the set of the equivalent operational prior parameters is now wider that means the importance of the prior distribution decreases with the sample size as one would expect.
   \begin{figure}[!ht]
  \centering
    \includegraphics[width=0.7\textwidth]{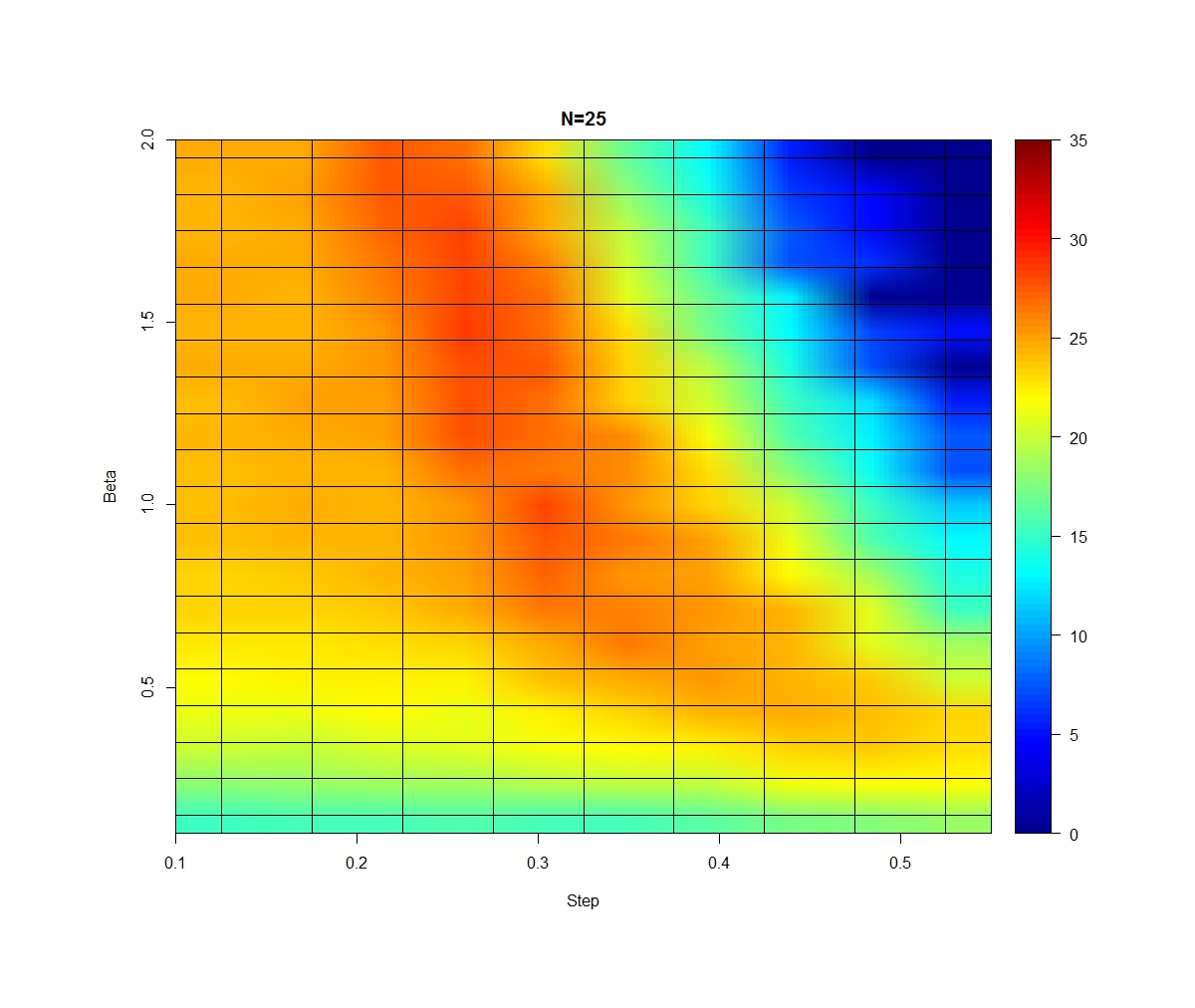}
         \caption{The geometric mean of the proportion of correct recommendations by the proposed WE method using different set of prior parameters: $\beta$ (vertical axis) and the difference between the risk of toxicity on the lowest and highest dose (horizontal axis) in six scenarios. Total sample size $N=25$. $10^6$ simulations are used.}
         \label{n25}
\end{figure}

\subsection{Safety constraint}
Parameters $r$ and $\gamma^*$ determine the strictness of the safety constraint. Greater values of $r$ and smaller values of $\gamma^*$ would lead to more conservative escalation. It helps to avoid high risk in unsafe scenarios, but also to prevent the correct recommendation in flat safe scenarios. There is a clear trade-off in the choice of these parameters that is precisely studied below.

Let us consider two extreme scenarios: a linear flat dose response shape with the target treatment far from the bottom ($d_5$) and the scenario with no safe treatment at all. The operating characteristics of the proposed method in these two scenarios with different parameters of the safety constraint $\gamma^*$ and $r$ are given in Table \ref{tab:tradeoff}. 

\begin{table}[ht!]
\caption{\textit{The operating characteristics of the proposed design in a linear and an unsafe scenario for different parameters of the safety constraint. The figures in the upper line of each cell corresponds to termination proportion in the unsafe scenario. The lower cell entries corresponds to the PCS in the flat linear scenario. The bold figures correspond to $\gamma^*=0.45$ and $r=0.035$ which were used in simulations. Results based on $10^6$ replications.\label{tab:tradeoff}}}
\centering
        \begin{tabular}{ccccccccc}
                      & \multicolumn{8}{c}{$r$} \\
            \cline{2-9}
             & $0.010$ & $0.015$ & $0.020$ & $0.025$ & $0.030$ & $0.035$ & $0.040$ & $0.045$ \\
            \hline
                        \multirow{2}{*}{$\gamma^*=0.55$}&  0.00 & 0.32 & 4.32 & 18.47 & 36.15 & 49.06 & 61.49 & 75.70 \\
                & 26.47 & 26.65 & 26.40 & 26.05 & 26.85 & 25.03 & 24.10 & 20.23\\
            \hline
             \multirow{2}{*}{$\gamma^*=0.50$}&0.15 & 2.50 & 17.76 & 38.75 & 52.74 & 63.06 & 74.94 & 87.22  \\
    & 26.27 & 26.22 & 26.53 & 27.24 & 25.46 & 23.30 & 20.35 & 17.10  \\
            \hline
             \multirow{2}{*}{$\gamma^*=0.45$}&1.13 & 12.72 & 35.72 & 56.49 & 67.16 & \textbf{77.55} & 86.53 & 93.49 \\
    & 26.15 & 26.02 & 26.81 & 25.18 & 24.26 & \textbf{23.15} & 18.16 & 11.05 \\
            \hline
             \multirow{2}{*}{$\gamma^*=0.40$}&  7.47 & 37.95 & 59.49 & 70.52 & 80.53 & 88.32 & 94.18 & 97.63 \\
    & 26.04 & 25.91 & 24.90 & 21.98 & 17.66 & 17.47 & 11.05 & 3.51 \\
       \hline
         \multirow{2}{*}{$\gamma^*=0.35$}& 33.98 & 58.22 & 74.42 & 84.14 & 90.52 & 94.86 & 97.90 & 99.20 \\
    &  25.65 & 24.54 & 20.45 & 15.55 & 13.77 & 9.21 & 6.25 & 0.70  \\
       \hline
        \multirow{2}{*}{$\gamma^*=0.30$}&    55.51 & 77.02 & 87.21 & 92.99 & 96.50 & 98.55 & 99.37 & 99.83 \\
    & 24.21 & 18.09 & 14.40 & 11.42 & 7.13 & 0.95 & 0.08 & 0.04 \\
       \hline
        \end{tabular}
\end{table}

The upper line in each cell corresponds to the proportion of times the proposed method declares that there is no safe dose when there is actually no safe dose. The lower line corresponds to the proportion of trials the actual TD25 was recommended in a linear scenario. The most relaxed safety constraint corresponds to the left upper corner. In this case no trials are terminated in a highly toxic scenario and the proportion of times the TD25 is recommended in the linear scenario is high. The right lower corner corresponds to the strictest safety constraint. In this case near all trials will be terminated when there is no safe dose, but the 
method will often not find the TD25 in the linear scenario. Therefore, the trade-off is to sacrifice the accuracy of the method when the TD25 is far from the bottom in order to prevent the recommendation of highly toxic dose in unsafe scenario.

\end{appendices}

\end{document}